\documentclass{amsart}

\usepackage[latin1]{inputenc}

\usepackage[all]{xy}
\usepackage{amssymb, amsmath, amscd, geometry}
\usepackage{graphicx}
\usepackage{setspace}
\usepackage[normalem]{ulem}

\numberwithin{equation}{section}
\input xy
\xyoption{all}

\usepackage{latexsym}
\usepackage[usenames]{color}
\usepackage{epic} 
\usepackage{mathrsfs}
\usepackage{euscript}
\usepackage{rotating}
\usepackage{verbatim}
\usepackage{tikz-cd}

\usepackage{ulem}

\newtheorem{lemma}{Lemma}[section]
\newtheorem{theorem}[lemma]{Theorem}

\newtheorem*{theorem*}{Theorem}

\theoremstyle{definition}

\newtheorem{remark}{Remark}[section]

      \newcommand{\C}{{\mathbb C}}

\newcommand{\h}{{H}}

\newcommand{\pl }{\,}
\newcommand{\lel }{\, =\, }
\newcommand{\kl}{\, \le \, }

\newcommand{\ten}{\otimes}

\newcommand{\norm}[1]{\left\lVert #1 \right\rVert}

\newcommand{\cbnorm}[1]{\left\lVert #1 \right\rVert_{\text{cb}}}

\newcommand{\uno}{1\!\!1}

\title[On a question of Blecher, Pisier, Shlyakhtenko]{On a question of\\ Blecher, Pisier, Shlyakhtenko}

\author{Roy Araiza} 
\address{Department of Mathematics \& Illinois Quantum Information Science and Technology Center\\University of Illinois at Urbana-Champaign\\1409 W. Green St. Urbana, IL 61891. USA}
\email{raraiza@illinois.edu}

\author{Marius Junge}
\address{Department of Mathematics \& Illinois Quantum Information Science and Technology Center\\University of Illinois at Urbana-Champaign\\1409 W. Green St. Urbana, IL 61891. USA}
\email{junge@math.uiuc.edu}

\author{Carlos Palazuelos}
\address{Instituto de Ciencias Matem\'aticas (ICMAT)\\Departamento de An\'alisis Matem\'atico y Matem\'atica Aplicada \\
Facultad de Ciencias Matem\'aticas \\ Universidad Complutense de Madrid \\
Madrid 28040. Spain}
\email{cpalazue@ucm.es}

\thanks{MJ is partially supported by the NSF grants DMS 2247114 and Raise-TAG183917. C.P. is partially supported by Grant PID2023-146758NB-I00 funded by MICIU/AEI/10.13039/501100011033, by TEC-2024/COM-84-QUITEMAD-CM funded by the Comunidad de Madrid, and by Grant CEX2023-001347-S funded by MICIU/AEI/10.13039/501100011033. C.P. gratefully acknowledges financial support for this publication from the Fulbright Program, sponsored by the U.S. Department of State and the U.S.-Spain Fulbright Commission.}

\addtocounter{tocdepth}{-1}

\begin{document}

\keywords{Grothendieck's Theorem, Operator spaces}

\maketitle

\begin{abstract}
We show the failure of a matricial version of Grothendieck's theorem for operator spaces, thereby resolving a long-standing open question in the field. Moreover, by showing that such a counterexample can occur in the simplest context of commutative $C^*$-algebras, we address some other open questions in operator algebras. Our constructions, completely explicit and fairly simple, are inspired by some techniques in quantum information theory.
\end{abstract}

\section{Introduction and main results}

Grothendieck's Resum\'e \cite{Grothendieck} concerning tensor norms of Banach spaces transformed the field of Banach space theory and led to the Grothendieck programme in the field of operator algebras. The Grothendieck programme has had a dramatic impact on various areas of mathematics, physics, and computer science, with various interpretations and applications being produced over many decades. The ``fundamental theorem'' in the Resum\'e concerns factorizability of operators between the universal injective Banach space $\ell_\infty$ and the universal projective Banach space $\ell_1$. In particular, Grothendieck's theorem may be stated as follow: 

\begin{theorem}[\cite{Grothendieck}]\label{thm: Grothendieck classic}
Given compact Hausdorff spaces $S,T$ and any bounded bilinear form $u: C(S) \times C(T) \to \mathbb C$ then there exist probability measures $\mu,\nu$ on $S$ and $T$, respectively, such that for every $(x,y) \in C(S) \times C(T)$ it follows \[
|{u(x,y)}| \leq K \|{u}\| \left( \int |x|^2 d\mu \right)^{\frac{1}{2}}\left( \int |y|^2 d\nu \right)^{\frac{1}{2}},
\] where $K$ is a numerical constant. 
\end{theorem}
{The previous theorem can be stated for general commutative C$^*$-algebras and, when applied to $\ell_\infty$, it implies that} every bounded linear operator $u: \ell_\infty \to \ell_1$ factors through a Hilbert space and thus satisfies finite $\gamma_2$-norm. {It can also be stated in terms of tensor norms as follows:}
 \begin{equation}
\ell_1 \otimes_\epsilon \ell_1 = \ell_1 \otimes_{\gamma_2^*} \ell_1, \label{GRO}
\end{equation} isomorphically, where here $\gamma_2^*$ represents the dual Hilbert space factorization norm. 

In his work, Grothendieck asked the question of replacing $\ell_1$ by noncommutative $L_1$-spaces. In particular, does the analogous identification in Equation~\eqref{GRO} hold if one considers duals of C*-algebras? In 1978 Pisier answered Grothendieck's question in the affirmative, with a refinement in 1985 due to Haagerup. 
\begin{theorem}[\cite{Pisier78,Haagerup85}]\label{thm: nc Grothendieck}
    Given two C*-algebras $A,B$ and any bounded bilinear form $u: A \times B \to \mathbb C$ there exist states $f_1,f_2 \in A^*, g_1,g_2 \in B^*$ such that for every $(a,b) \in A \times B$ it follows \[
|u(a,b)| \leq  \norm{u} \{f_1(aa^*) + f_2(a^*a)\}^{\frac{1}{2}}\{g_1(bb^*) + g_2(b^*b)\}^{\frac{1}{2}}.
    \] 
\end{theorem} Analogous to the dual statement in Equation~\eqref{GRO}, Theorem~\ref{thm: nc Grothendieck}, referred to as the \emph{noncommutative Grothendieck Theorem}, says \begin{equation}
A^* \otimes_\epsilon B^* = A^* \otimes_{\tilde{\gamma}_2^*} B^*, \label{eq: dual nc Grothendieck}
\end{equation} isomorphically as Banach spaces, where $\tilde{\gamma}^*_2$ is the analog to the $\gamma_2^*$ norm in the context of operator algebras. {As before, this theorem implies that every bounded operator $T:A\rightarrow B^*$ factors through a Hilbert space}.

The emergence of the field of operator space theory in the 90's then led to the question of how the Grothendieck programme translated to this new regime of objects. Concretely, \emph{operator spaces} are closed subspaces of C*-algebras endowed with an additional matricial structure. The proper morphisms of bounded linear operators are replaced with the class of completely bounded operators. One may consider Banach spaces with the additional structure coming from an isometric embedding into some $ \mathcal B(H)$. The theory of operator spaces was greatly developed during the 90's to which we refer the interested reader to \cite{EffrosRuanBook,Pisierbook}. The category of operator spaces is closed under many of the familiar constructions one encounters in Banach space theory. In particular, duals of operator spaces are again operator spaces. The question why we would even consider Grothendieck's fundamental theorem for operator spaces then becomes extremely well-motivated since Banach space duals of C*-algebras as in Equation~\eqref{eq: dual nc Grothendieck} are not necessarily C*-algebras, but they are operator spaces. Given C*-algebras $A \subset  \mathcal B(H)$ and $B \subset  \mathcal B(K)$, then the minimal C*-algebra tensor product of $A$ and $B$, denoted $A \otimes_{{min}} B$ is induced via the natural embedding \[A \otimes B \subset  \mathcal B(H \otimes_2 K).\] The operator space minimal tensor product is induced in exactly the same fashion. If we take two operator spaces $X \subset  \mathcal B(H)$ and $Y \subset  \mathcal B(K)$, then the minimal operator space tensor product of $X$ and $Y$, denoted $X \otimes_{{min}} Y$, is induced via the natural embedding \[X \otimes Y \subset  \mathcal B(H \otimes_2 K).\] The appropriate operator space tensor to replace the right-hand side of Equation~(\ref{eq: dual nc Grothendieck}) is the \emph{Haagerup tensor product} $\otimes_h$, which has no analogue in classical Banach space theory (\cite{BlPa91}). Though it has many equivalent formulations, in keeping with the motif of factorizability, given two operator spaces $X, Y$ and a linear operator $u: X \to Y$, then we say that $u$ is \emph{column-factorizable} if there exist completely bounded maps $r: X \to C, s: C \to Y$ such that $u = sr$. Here, we let $C \subset \mathcal B(\ell_2)$ denote the \emph{column operator space} which is concretely realized as \[
C:= \overline{\text{span}}\{e_{i1}: i \in \mathbb N\}.
\]Denoting $\gamma_c(u):= \inf \{\cbnorm{r}\cbnorm{s}: u = sr\}$, then we consider the {Banach} space $\Gamma_c(X, Y)$ of all linear operators $u: X \to Y$ such that $\gamma_c(u) < \infty$ {which, in fact, comes with a natural operator space structure}. We then define the Haagerup tensor product of $X$ and $Y$, denoted $X \otimes_h Y$, as the operator space induced via the embedding \[X \otimes Y \subset \Gamma_c(Y^*,X).\] Despite the many amazing properties of $\otimes_h$ including associativity, injectivity, projectivity and self-duality, the Haagerup tensor product is not {commutative}. Thus, if one is to consider the full analogue of Equation~(\ref{eq: dual nc Grothendieck}) then the correct candidate is the \emph{symmetrized Haagerup tensor product} of $X$ and $Y$: \[X \otimes_\mu Y = X \otimes_h Y + X \otimes_{h^t} Y,\]
where $h^t$ denotes the transposed Haagerup tensor $Y \otimes_h X$. 
\newline \indent In 1991 \cite{Blecher} Blecher conjectured a fully functorial version of the \emph{operator space Grothendieck theorem}: for general C*-algebras $A,B$ it follows \begin{equation}
A^* \otimes_{{min}} B^* = A^* \otimes_\mu B^*, \label{Blecher conjecture}
\end{equation} completely isomorphically. Due to a groundbreaking result of Pisier-Shlyakhtenko in 2002, with follow-up work from Haagerup-Musat in 2008, we have the following:
\begin{theorem}[\cite{PiSh02,HaMu08}]
Given two C*-algebras\footnote{Notably, this result extends beyond the setting of C* -algebras to encompass general exact operator spaces; however, in this case, one needs to add a constant to the right-hand side of the inequality \cite{PiSh02}.} $A,B$ and any jointly completely bounded bilinear form $u: A \times B \to \mathbb C$, then there exist states $f_1,f_2 \in A^*, g_1,g_2 \in B^*$ such that for every $(a,b) \in A \times B$ it follows \[
|u(a,b)| \leq \|u\|_{jcb}\left(\{f_1(aa^*)g_1(b^*b)\}^{\frac{1}{2}} + \{f_2(a^*a)g_2(bb^*)\}^{\frac{1}{2}}\right).
\]  
\end{theorem}  A bilinear form $u: X \times Y \to \mathbb C$ on the operator spaces $X,Y$ is said to be \emph{jointly completely bounded} if the induced map $u: X \to Y^*$ is completely bounded.  {Under this assumption, the preceding theorem shows that Equation~\eqref{Blecher conjecture} holds isomorphically. In recent years, new proofs of this result have appeared, based on techniques from quantum information theory \cite{regev2014elementary}, reflecting a deepening understanding of the problem. Nevertheless, despite these significant advances, and the considerable efforts already mentioned, the stronger question of whether Equation~\eqref{Blecher conjecture} held completely isomorphically remained unresolved (see \cite{PiSh02} and \cite[Problem 21.2]{Pi12}).}

Determining the uniqueness of C*-tensor norm completions of the algebra $A \otimes B$ is intimately connected with the Grothendieck programme. Such identifications are at the core of notions such as exactness, nuclearity, the weak expectation property (WEP), and the local lifting property (LLP) ({see \cite{Pisierbook} for a very complete reference on the topic}). Indeed, in 1993 \cite{kirchberg1993non} Kirchberg proved that the Connes' embedding problem was equivalent to the full group C*-algebra of the free group $F_\infty$ having WEP. In other words, $C^*(F_\infty) \otimes C^*(F_\infty)$ admits only one C*-tensor completion. Preceding Pisier-Shlyakhtenko's 2002 result, partial success was made by Junge-Pisier in \cite{junge1995bilinear} in which they proved that given exact operator spaces $X, Y$, and a completely bounded map $u: X \to Y^*$, then $u$ factors boundedly through a Hilbert space with upper bound ($\leq 4\cbnorm{u})$. This was their strategy in proving that the algebra $\mathcal B(H) \otimes \mathcal B(H)$, for infinite-dimensional $H$, admits {more than one} C*-completion. In particular, the universal injective C*-algebra $\mathcal B(H)$ does not have the LLP. We will come back to the question of $C^*(F_\infty)$ having WEP at the end of this section.

Our first main result is a counterexample to Blecher's conjecture. To produce this example, a full matrix structure is not needed, but only a column structure. 
\begin{theorem}\label{thm1}
$C \otimes_{\text{min}} (S_\infty(\ell_2)^* \otimes_{\text{min}} S_\infty(\ell_2)^*) \neq C \otimes_{\text{min}} (S_\infty(\ell_2)^* \otimes_{\mu} S_\infty(\ell_2)^*)$ isomorphically.
\end{theorem}
{In fact}, we prove an even stronger result by  replacing the operator space $S_\infty(\ell_2)^* \otimes_{{min}} S_\infty(\ell_2)^*$ by $S_\infty(\ell_2)^* \otimes_{max} S_\infty(\ell_2)^*$. In the context of C*-algebras $A,B$, the \emph{maximal C*-algebra} tensor product is obtained via completing with respect to the C*-tensor norm \[
\norm{u}_{max}:= \sup\{ \norm{\pi(u)}_{\mathbb B(H_\pi)}: \pi: A \otimes B \to \mathbb B(H_\pi),\,\,\text{a $*$-representation}\}.
\] The \emph{maximal tensor product} of the operator spaces $X,Y$ is induced via the embedding of the operator spaces into their respective \emph{{unital} universal C*-algebras} (see \cite{Ozawa}) \begin{equation}
    X {\otimes _{max}}Y \subset C^*\langle X \rangle \otimes_{max} C^* \langle Y \rangle.
\end{equation} Interestingly enough, such a characterization for $\otimes_\mu$ also occurs. Proven in \cite{Oi99}, it follows \begin{equation}
    X \otimes_\mu Y \subset OA(X ) \otimes_{max} OA (Y),
\end{equation} completely isometrically, where here $OA (\cdot)$ denotes the \emph{unital universal operator algebra} of the corresponding operator space (see Section~\ref{sec: preliminaries}).

The previous viewpoint strongly motivates considering Equation~\eqref{Blecher conjecture} when focusing on commutative $C^*$-algebras; specifically, when  $X=Y=\ell_1$. This problem, already considered in \cite{PiSh02} (see also \cite[Chapter 25]{Pisierbook}), highlights the equivalence, at the matrix level, to the context originally considered by Grothendieck, since the equivalence $$ \ell_1\otimes_{\mu}\ell_1 \cong \ell_1\otimes_{max}\ell_1\quad \mbox{isomorphically}  \pl,$$ follows from the classical Grothendieck Theorem stated in Equation (\ref{GRO}). Our second result, the main result of this work, states that this is no longer true in the category of operator spaces:
 \begin{theorem}\label{thm2} 
  \begin{equation*}
  S_\infty(\ell_2) \ten_{\text{min}} (\ell_1 \ten_{max} \ell_1)
 \neq 	S_\infty(\ell_2) \ten_{\min} (\ell_1 \ten_{\mu} \ell_1) \quad \text{isomorphically} \pl.
  \end{equation*}
 \end{theorem}
 
One can alternatively express Theorem \ref{thm2} by indicating that the natural inclusion \[OA(\ell_1)\otimes_{max} OA(\ell_1)\subseteq C^*\langle \ell_1\rangle \otimes_{max} C^*\langle \ell_1\rangle,\] is not completely isomorphic, even when restricted to the linear subspace $\ell_1 \ten \ell_1$. According to the precise definitions of $OA(\ell_1)$ and $C^*\langle \ell_1\rangle$, Theorem \ref{thm2} underscores a significant limitation involving commuting contractions and commuting  unitaries, or even commuting normal contractions (see Remark \ref{rem_commuting}). In particular, Theorem~\ref{thm2} directly links the Grothendieck programme with fundamental questions in operator algebras. 
\newline \indent  A counterexample to the question of $C^*(F_\infty)$ having WEP was {recently} proposed in \cite{MIP RE} in which the use of \emph{quantum games} was paramount in their argument. Independent of questions regarding Connes' embedding problem, quantum games were precisely the inspiration in the proofs of our main results. Both classical and quantum games have had an outsized influence in establishing further connections between quantum theory and operator algebras. To both classical games and quantum games are associated \emph{game tensors}, which through the work of many authors (see \cite{Pa16} and the references therein) were found to determine various game values, and therefore the quantum protocols that such strategies of games described. In the classical setting, tensors for two-prover games belong to spaces of the form $\ell_1(\ell_\infty) \otimes \ell_1(\ell_\infty)$, where in the quantum setting the corresponding space is $S_1(S_\infty) \otimes S_1(S_\infty)$. Note that the commutative C*-algebra $\ell_\infty$ has a natural operator space structure, with $\ell_1$ have the dual operator space structure. In particular, both tensor products may be endowed with Banach space or operator space tensor products (as in Theorem~\ref{thm2}). If one considers the protocols described via game tensors in the space \[\ell_1(\ell_\infty) \otimes_\epsilon \ell_1(\ell_\infty), \quad \text{or} \quad \ell_1(\ell_\infty) \otimes_{{min}} \ell_1(\ell_\infty),\] then one arrives at the deep connections between classical and quantum values of classical games. Fundamental progress in quantum games was made in \cite{Co15} where the authors provide such characterizations for rank-one quantum games. Recent work \cite{Ara,Cra} has extended our understanding of quantum games in which the \emph{quantum value} of a quantum game is determined by evaluating the norm of the game tensor in the space $C \otimes_{{min}} (S_1 \otimes_{{min}} S_1)$. This protocol corresponds to the scenario where the two-players of the game are allowed resources in the form {of} states on auxiliary environments (Hilbert spaces).
\newline \indent Our main results not only add another chapter to the Grothendieck programme, but also further establish this deep connection between operator algebras and quantum information theory. Indeed, inspiration in our proofs of Theorem~\ref{thm1} and Theorem~\ref{thm2} were the generalized CHSH$_n$ games of Bavarian and Shor \cite{Ba15}. The tensor $\eta_n \in C \otimes S_1 \otimes S_1$ considered in the proofs of both stated theorems is correlated with this game formulation. In order to make the leap from tensors in $C \otimes S_1 \otimes S_1$ to $S_\infty \otimes \ell_1 \otimes \ell_1$ we make use of techniques of the second two authors in \cite{Ju16}, whose results were inspired by the quantum teleportation protocol from quantum information theory. Further interpretations of quantum protocols by operator space tensor norms brings one to fascinating lines of inquiry. What are the protocols determined by the operator space tensor norms appearing in Equation~\eqref{Blecher conjecture}? What are the rules governing the way in which our two-players may communicate (if at all) with one-another? Such questions and answers are the focal point of \cite{Ara}.

\section{preliminaries}\label{sec: preliminaries}

\subsection{Basic on operator spaces and operator algebras}

In this section we will review some of the necessary basics of operator spaces and their tensor products which we will invoke throughout the manuscript.  We refer to \cite{EffrosRuanBook, Pisierbook} for comprehensive references on the topic.

An \textit{operator space} is a closed subspace $X \subset \mathcal B(H)$, where $\mathcal B(H)$
denotes the spaces of bounded operators acting on a complex Hilbert space $H$. For any such subspace one naturally obtains a sequence of matrix norms $\|\cdot\|_n$ on $M_n(X)$ via the inclusion $M_n(X) \subseteq M_n(\mathcal B(H))\simeq \mathcal B(H^{\oplus n})$. Ruan's Theorem \cite{Ru88} characterizes the sequences of norms that can be achieved in this manner, offering an alternative definition of an operator space as a complex Banach space $X$ equipped with a sequence of matrix norms $(M_n(X),\|\cdot\|_n)$ that satisfy specific conditions. In both scenarios, we say that  $X$ is endowed with an \textit{operator space structure}.

When considering operator spaces, the norm on linear operators must reflect the matrix structure defined by the operator space structure. For a linear map between operator spaces $T:X\rightarrow Y$, we say that it is  \textit{completely bounded} if $\|T\|_{cb}:=\sup_n\|\text{Id}\otimes T:M_n(X)\rightarrow M_n(Y)\|$ is finite. We will say that $T$ is  \textit{completely contractive} if $\|T\|_{cb}\leq 1$. Moreover, $T$ is termed a  \textit{complete isomorphism} (resp.  \textit{complete isometry}) if it is an isomorphism with $T$ and $T^{-1}$ completely bounded (resp. completely contractive).

Given an operator space $X$, then $X^*$ has a privileged operator space structure with norms on $M_n(X^*)$ given via the identification \[M_n(X^*)  =  \text{CB}(X, M_n),\] where the latter denotes the operator space of completely bounded maps $T: X \to M_n$ endowed with the norm $\|\cdot\|_{cb}$.  Every $C^*$-algebra $A$ has a natural operator space structure. Therefore, by duality, $\ell_1$ and the space of trace class operators $S_1$ also have a natural operator space structure, as they are the duals of $c_0$ and the space of compact operators $\mathcal K(\ell_2)$, respectively. Moreover, the natural operator spaces structure on the space of bounded operators acting on $\ell_2$, $S_\infty$, matches the one given by the duality $S_1^*=S_\infty$. As explained in \cite{EffrosRuanBook}, the right duality pairing here is given by $\langle A, B\rangle=\operatorname{tr}(AB^T)$, where $\operatorname{tr}$ denotes the standard trace.

Throughout the text, $M_n$ and $S_\infty^n$ denote the same (operator) space. We usually use the notation $M_n$ when we want to emphasize the algebraic structure of the space and the operator space structure is not relevant. However, following the standard notation in operator space theory \cite{EffrosRuanBook, Pisierbook}, this notation is also used to write $M_n(X)$ and $M_n \otimes_{\min} X$, where the operator space structure is crucial.

Important examples which we use throughout this manuscript are the  \textit{column}  and  \textit{row} Hilbert operator spaces. Given a Hilbert space $H$, the column operator space $H_c$ is defined by means of the natural identification $$H\cong \mathcal B(\C, H).$$ For the particular case $H=\ell_2^n$, we will simply denote $C_n$. By duality, we can define the row operator space $H_r=(H_c)^*$, where we recall that, given a Hilbert space $H$, its dual space $H^*$ is naturally identified with the conjugate $\overline{H}$.   As in the previous case, the finite dimensional row space will be simply denoted by $R_n$.

An abstract theory for operator space tensor products was first developed by Blecher and Paulsen in \cite{BlPa91} in which they introduce the \textit{projective} and  \textit{injective} operator space tensor products. Given two operator spaces $X, Y$, and $w \in M_n(X \otimes Y)$,  define the \textit{operator space projective norm} of $w$ as \begin{align*}
\| w\|_{\wedge}:= \inf\{\|\alpha\| \|x\| \|y\| \|\beta \|\},
\end{align*}where the infimum is taken over $p,q,n\in \mathbb N$, $x \in M_p(X)$, $y \in M_q(Y)$, $\alpha \in M_{n,pq}$, $\beta \in M_{pq,n}$, such that $w = \alpha(x \otimes y)\beta$. We denote by $X \hat{\otimes} Y$ the corresponding operator space. As one may expect from its name, the projective tensor norm preserves  (complete) quotients. One may also consider this operator space as the ``predual'' of the operator space $\text{CB}(X, Y^*)$. In particular, one has the complete isometry $(X \hat{\otimes} Y)^* = \text{CB}(X, Y^*)$. As previously stated, the operator space \textit{injective (minimal)} tensor product is induced via the embedding into the minimal C*-algebra tensor product. Thus, if $X \subset \mathcal B(H), Y \subset \mathcal B(K)$, then we have the completely isometric embedding $X \otimes_{min} Y \subset \mathcal B(H \otimes_2 K)$. In particular, it follows that the min norm is preserved by (complete) isometries. It can be seen that, given $w \in M_n(X \otimes Y)$, one has $$\|w\|_{M_n(X\otimes_{min} Y)}=\sup \Big\|\big(\text{Id}\otimes T\otimes S\big)(w)\Big\|_{M_n(\mathcal B(H\otimes K))},$$where the supremum is taken over all Hilbert spaces $H$ and $K$ and all complete contractions $T:X\rightarrow \mathcal B(H)$ and $S:Y\rightarrow \mathcal B(K)$. In comparison to the projective tensor product, we have the completely isometric embedding $X\otimes_{min}Y\subseteq CB(X^*,Y).$ Hence, we see that the projective and the injective norms are dual to each other and for finite-dimensional operator spaces  $X$, $Y$, there are the following completely isometric identifications $$(X \hat{\otimes} Y)^*=X^* \otimes_{min} Y^*, \hspace{0.4 cm} (X \otimes_{min} Y)^*=X^*  \hat{\otimes}Y^*.$$

Given operator spaces $X \subset \mathcal B(H)$ and $Y \subset \mathcal B(K)$, we define the \textit{Haagerup operator space norm} for $w \in M_n(X \otimes Y)$  as \[
\|w\|_h:= \inf \{ \|u\|_{M_{n,r}(X)} \|v\|_{M_{r,n}(Y)}\},
\] where the infimum runs over $r\in \mathbb N$ and $u$, $v$, such that $w_{i,j}=\sum_{k=1}^r u_{i, k}\otimes v_{k,j}$ for every $i,j$. We denote by $X \otimes_h Y$ the corresponding operator space, which is usually called the \textit{Haagerup operator space tensor product}. Despite being injective, projective and self-dual, it has been long known that $\otimes_h$ is not commutative. 

It is well known that the previous norms behave nicely when considering tensor products of linear maps. Specifically, if 
$\alpha$ denotes any of the previous three norms and $T_i:X_i\rightarrow Y_i$ are completely bounded maps for $i=1,2$, then the map $$T_1\otimes T_2:X_1\otimes_\alpha X_2\rightarrow Y_1\otimes_\alpha Y_2$$is completely bounded and $$\|T_1\otimes T_2\|_{cb}=\|T_1\|_{cb}\|T_2\|_{cb}.$$This is usually refer to as the \textit{metric mapping property} in the category of operator spaces. Also, by iterating the process, one can define the aforementioned norms on the tensor product of more than two spaces. Remarkably, all three norms exhibit associativity. Furthermore, the previous properties (related to projectivity, injectivity,  duality, etc.) remain consistent and apply in the same manner.

The previously discussed tensor norms are closely related when examining the column and row spaces. In fact, it can be proven that for every operator space $X$, the equalities
\begin{align}
C\otimes_{min} X=C\otimes_h X\hspace{0.3 cm}\text{and} \hspace{0.3 cm}X\otimes_{min} R=X\otimes_h R
\end{align}hold completely isometrically. By duality, one also has $R\hat{\otimes} X=R\otimes_h X$ and $X\hat{\otimes}C=X\otimes_h C$ completely isometrically.

The previous identifications, along with the aforementioned properties of the tensor norms we are considering, easily imply that the following identities are complete isometries:
\begin{align}\label{Identities_S_infty_S_1}
S_\infty\otimes_{min} X = C \otimes_h X \otimes_h R, \hspace{0.3 cm}S_1 \hat{\otimes} X = R \otimes_h X \otimes_h C.
\end{align}

Equation (\ref{Identities_S_infty_S_1}) and the well-known estimate $\|\text{Id}:R_n\rightarrow C_n\|_{cb}=\|\text{Id}:C_n\rightarrow R_n\|_{cb}=\sqrt{n}$, can be used to show that for every operator space $X$, it holds: 
\begin{align}\label{Id_com_int}
\Big\|\frac{\text{Id}}{n}:S_\infty^n\otimes_{min} X\rightarrow S_1^n\hat{\otimes }X\Big\|_{cb}=1.
\end{align}

As emphasized in the introduction, the non-commutativity of the Haagerup tensor norm leads one to consider the so-called \textit{symmetrized Haagerup tensor product} $\otimes_\mu$, where given $w \in M_n(X \otimes Y)$ we define
\begin{align}\label{Def_mu_norm}
\| w\|_{\mu}:= \inf \{ \|w_1\|_{M_n(X \otimes_h Y)} + \|w_2^T\|_{M_n(Y \otimes_h X)}: w = w_1 + w_2\}.
\end{align}Here, for $w=\sum_ia_i\otimes x_i\otimes y_i$, we have let $w_2^T=\sum_ia_i\otimes y_i\otimes x_i$.  

Given an operator space $X$, its universal (unital) operator algebras $OA(X)$ (see \cite[Chapter 6]{Pisierbook} for a detailed construction of the algebra) is characterized by the following universal property: for any complete contraction $T:X\rightarrow \mathcal B(H)$ there exists a unique unital morphism $\pi:OA(X)\rightarrow \mathcal B(H)$ extending $T$, where $X$ is naturally embedded in $OA(X)$. This algebra is particularly interesting since, for operator spaces $X$ and $Y$, it is known \cite[Lemma 5]{Oi99} that
\begin{align*}
X\otimes_{\mu} Y\subset OA(X)\otimes_{max} OA(Y)\hspace{0.3 cm}\text{completely isometrically.}
\end{align*} Here, the max norm is defined at the level of unital operator algebras. In particular, the previous inclusion shows that for any element $w\in M_n(X\otimes Y) $, we have
\begin{align}\label{Defintion_mu_norm}
\|w\|_{M_n(X\otimes_{\mu} Y)}=\sup \Big\|\big(\text{Id}\otimes T\odot S\big)(w)\Big\|_{M_n(\mathcal B(H))},
\end{align}where the supremum runs over all complex Hilbert spaces $H$ and all complete contractions $T:X\rightarrow \mathcal B(H)$ and $S:Y\rightarrow \mathcal B(H)$ with commuting ranges and where we have denoted $T\odot S$ the operator such that $T\odot S(x\otimes y)=T(x)S(y)$.

On the other hand, the universal (unital) $C^*$-algebra associated to an operator space $X$, $C^*\langle X\rangle$ (see \cite[Chapter 8]{Pisierbook} for a detailed construction of the algebra\footnote{This algebra should not be confused with the injective envelop of an operator space $I(X)$ as discussed in \cite{BlPa01, Ha79}.}), is characterized by the following universal property: For any complete contraction $T:X\rightarrow \mathcal B(H)$ there exists a unique unital $*$-homomorphism $\pi:C^*\langle X\rangle\rightarrow \mathcal B(H)$ extending $T$, where $X$ is naturally embedded in $C^*\langle X\rangle$. It can be seen that $OA(X)$ can be  identified with the closed subalgebra generated by $X$ in $C^*\langle X\rangle$ so that $OA(X)\subset C^*\langle X\rangle$ completely isometrically. One then defines the max tensor product of two operator spaces $X$ and $Y$ so that the following inclusion is a complete isometry: $$X\otimes_{max}Y\subset C^*\langle X\rangle\otimes_{max} C^*\langle Y\rangle,$$where $C^*\langle X\rangle\otimes_{max} C^*\langle Y\rangle$ denotes the corresponding tensor product in the category of $C^*$-algebras. In particular, if we use bracket notation $[A,B]$ for the commutator of two operators in $\mathcal B(H)$, for any element $w\in M_n(X\otimes Y)$ we have \begin{align}\label{Def_max_norm}
\|w\|_{M_n(X\otimes_{max} Y)}=\sup \Big\|\big(\text{Id}\otimes T\odot S\big)(w)\Big\|_{M_n(\mathcal B(H))},
\end{align}where the supremum is taken over all complex Hilbert spaces $H$ and all complete contractions $T:X\rightarrow \mathcal B(H)$ and $S:Y\rightarrow \mathcal B(H)$ such that $[T(x), S(y)]=[T(x), S(y)^\dag]=0$ for every $x\in X$ and $y\in Y$. Here and throughout the text, we use the notation $T^\dag$ to denote the adjoint of the operator $T$ in $\mathcal B(H)$ of $M_n$. This contrasts with the notation $T^*$, which is commonly used in the context of C$^*$-algebras.

Note that it follows immediately from Eq. (\ref{Defintion_mu_norm}) and  Eq. (\ref{Def_max_norm}) that for every $w \in M_n(X \otimes Y)$, we have
\[
\|w\|_{M_n(X \otimes_{\max} Y)} \leq \|w\|_{M_n(X \otimes_{\mu} Y)}.
\]

\subsection{Some basic results involving the Haagerup tensor norm}

In this section, we present some auxiliary results that will significantly simplify the readability of the proofs for the main theorems.
\begin{lemma}\label{Lemma_Aux_I}
Given any natural number $n$, we have $$\|\text{Id}: S_1^n\hat{\otimes}(S_1^n\otimes_{min} S_\infty^n)\rightarrow  S_1^{{n^2}}\otimes_{min} S_\infty^n\|_{cb}\leq 1.$$
\end{lemma}
\begin{proof}
According to the properties of the min and the Haagerup tensor norms, the following identifications are complete isometries.
\begin{align*}
S_1^n\hat{\otimes}(S_1^n\otimes_{min} S_\infty^n)=S_1^n\hat{\otimes}(S_\infty^n\otimes_{min} S_1^n)&= R_n^A\otimes_h(C_n^B\otimes_h(R_n^C\otimes_h C_n^C)\otimes_h R_n^B)\otimes_hC_n^A\\&=(R_n^A\otimes_hC_n^B)\otimes_hR_n^C\otimes_h C_n^C\otimes_h (R_n^B\otimes_hC_n^A).
\end{align*}Note that we have used the super indexes $A$, $B$ y $C$ to denote the three spaces appearing in the tensor product $S_1^n\hat{\otimes}(S_\infty^n\otimes_{min} S_1^n)$. This allows us to indicate to which of these spaces the column or row spaces correspond in the following equalities.

Now, since $\|\text{Id}: R_n^A\otimes_hC_n^B\rightarrow C_n^B\otimes_hR_n^A\|_{cb}\leq 1$ and $\|\text{Id}: R_n^B\otimes_hC_n^A\rightarrow C_n^A\otimes_h R_n^B\|_{cb}\leq 1$, we conclude our proof by noticing the following completely isometric identifications:
\begin{align*}
(C_n^B\otimes_hR_n^A)\otimes_hR_n^C\otimes_h C_n^C\otimes_h (C_n^A\otimes_h R_n^B)&=C_n^B\otimes_h(R_n^A\otimes_hR_n^C\otimes_h C_n^C\otimes_h C_n^A)\otimes_h R_n^B\\&=C_n^B\otimes_h(R_{{n^2}}^{AC}\otimes_h C_{{n^2}}^{AC})\otimes_h R_n^B
\\&= S_\infty^n(S_1^{{n^2}}).
\end{align*}
\end{proof}

The following lemma is a direct consequence of the fact that the product on $M_n$ defines a complete contraction $\|M_n(X)\otimes_h M_n(Y)\rightarrow M_n(X\otimes_hY)\|_{cb}\leq 1$ (which easily follows from \cite[Theorem 5.15]{Pisierbook}). For completeness, we include a simple proof of the specific version that we will use here.

\begin{lemma}\label{Lemma_Aux_II}
Given two operator spaces $X$ and $Y$ and any natural number $n$, we have
\begin{enumerate}
\item[a)] $\|T:X\otimes_hM_n(Y)\rightarrow M_n(X\otimes_hY)\|_{cb}\leq 1$, where $T$ is defined as $T(x\otimes A\otimes y)=A\otimes x\otimes y$ on elementary tensors.
\item[b)] $\|\text{Id}:M_n(X)\otimes_hY\rightarrow M_n(X\otimes_hY)\|_{cb}\leq 1.$
\end{enumerate}
\end{lemma}
\begin{proof}
In order to prove the first estimate, note that $$X\otimes_hM_n(Y)=(X\otimes_h (C_n\otimes_h Y\otimes_h R_n))=(X\otimes_h C_n)\otimes_h (Y\otimes_h R_n)$$ completely isometrically. Now, $\|\text{Id}:X\otimes_h C_n\rightarrow C_n\otimes_h X\|_{cb}\leq 1$. On the other hand,  $$(C_n\otimes_h X)\otimes_h (Y\otimes_h R_n)=C_n\otimes_h (X\otimes_h Y)\otimes_h R_n=M_n(X\otimes_h Y)$$completely isometrically. Hence, we prove the estimate a).

In order to prove the second estimate, note that $$M_n(X)\otimes_hY= (C_n\otimes_h X\otimes_h R_n)\otimes_h Y=(C_n\otimes_h X)\otimes_h (R_n\otimes_h Y)$$completely isometrically. Now, $\|\text{Id}:R_n\otimes_h Y\rightarrow Y\otimes_h R_n\|_{cb}\leq 1$. On the other hand,  $$(C_n\otimes_h X)\otimes_h (Y\otimes_h R_n)=C_n\otimes_h (X\otimes_h Y)\otimes_h R_n=M_n(X\otimes_h Y)$$completely isometrically. Hence, estimate b) follows.
\end{proof}

 \begin{lemma}\label{lemma_haag_P_1}
 Let $X$ be an operator space and $(x_k)_k\subset X$ any sequence. Then, $$\Big\|\sum_k e_k\otimes x_k\otimes e_k\Big\|_{R\otimes_hX\otimes_hR}=\big(\sum_k \|x_k\|^2_{X}\Big)^{\frac{1}{2}}.$$
  \end{lemma} 
 \begin{proof}
 Consider the operator space $Y=R\otimes_hX\subset \mathcal B(H)$, for a certain Hilbert space $H$ and denote $y_k=e_k\otimes x_k\in \mathcal B(H)$ for every $k$. It is clear that $\|y_k\|=\|x_k\|$ for every $k$. In addition, using the injectivity of the Haagerup tensor norm, we can write 
 \begin{align*}
 \Big\|\sum_k y_k\otimes e_k\Big\|^2_{R\otimes_hX\otimes_h R}=\Big\|\sum_k y_k\otimes e_k\Big\|^2_{\mathcal B(H)\otimes_h R}= \Big\|\sum_k y_ky_k^*\Big\|_{\mathcal B(H)}\leq \sum_k \|y_k\|^2_{\mathcal B(H)} =\sum_k \|x_k\|^2_{X}, 
  \end{align*}where we have used the triangle inequality.

In order to prove the converse inequality, note that a completely analogous proof applies if we replace $R$ with $C$ and $X$ with $X^*$, so that we can prove that for every sequence $(x_k^*)_k\subset X^*$ the following inequality holds:
$$\Big\|\sum_k e_k\otimes x_k^*\otimes e_k\Big\|^2_{C\otimes_hX^*\otimes_h C}\leq \sum_k \|x_k^*\|^2_{X^*}.$$

Hence, since $(R\otimes_hX\otimes_hR)^*=C\otimes_hX^*\otimes_h C$ (completely) isometrically and also $(\ell_2(X))^*= \ell_2(X^*)$ isometrically, the inequality $$\Big(\sum_k \|x_k\|^2_{X}\Big)^{\frac{1}{2}}\leq \Big\|\sum_k e_k\otimes x_k\otimes e_k\Big\|_{R\otimes_hX\otimes_hR},$$follows by duality.
 \end{proof}
 
 \begin{lemma}\label{lemma_haag_P_2}
 Let $X$ be an operator space and $(x_k)_k\subset X$ any sequence. Then, $$\Big\|\sum_k e_k\otimes x_k\otimes e_k\Big\|_{R\otimes_hX\otimes_hC}=\sum_k \|x_k\|_{X}.$$
 \end{lemma} 
 \begin{proof}
 The statement follows from the isometric identification $R\otimes_hX\otimes_hC=S_1(X)$, simply by restricting to the diagonal of $S_1$, which is $\ell_1$.
 \end{proof}

\section{Proof of Theorem \ref{thm1}}

In this section we prove the following theorem, from which Theorem \ref{thm1} follows immediately.
\begin{theorem}\label{Thm_Main_I}
For every prime number $n$ there exists an element $\eta_n\in C_{n^4}\otimes S_1^n\otimes S_1^n$ such that $\|\eta_n\|_{C_{n^4}\otimes_{min} (S_1^n\otimes_{max} S_1^n)}\leq \sqrt{2}n^{\frac{3}{4}}$ and $\|\eta_n\|_{C_{n^4}\otimes_{min} (S_1^n\otimes_{\mu} S_1^n)}\geq  n$.
\end{theorem}

Here and throughout the rest of the paper, we are considering $\mathbb Z_n=\{0,1,\ldots , n-1\}$ with addition and multiplication modulo $n$. For consistency in notation, we will represent the canonical basis of $\C^n$ as $\{e_i:\, i=0,1,\ldots, n-1\}$. Additionally, we use the notation $e_{i,j}$ to denote the matrix whose entries are all zero except for a one in the $i$-th row and $j$-th column. We recall that the mapping class $\text{CB}(S_1^n: \mathcal B(H))$ may be identified with $S_\infty^n \otimes_{min} \mathcal B(H)$, or equivalently, with the $n \times n$ matrices over $\mathcal B(H)$. 

The upper bound in the previous result will follow from the next two lemmas. 
  \begin{lemma}\label{lemma_basic_1}
  Let $T:S_1^n\rightarrow \mathcal B(H)$ be a complete contraction and denote $E_x^a=T(e_{a,x})^\dag T(e_{a,x})$ for every $x, a=1,\cdots, n$. Then, for every $x$ and every sequence of complex numbers $(\alpha_a)_a$ such that $\sup_a|\alpha_a|\leq 1$, it holds that $$\Big\|\sum_a\alpha_a E_x^a\Big\|_{\mathcal B(H)}\leq 1.$$
  \end{lemma}
  \begin{proof}
According to the isometric identifications $CB(S_1^n, \mathcal B(H))=S_\infty^n\otimes_{min}\mathcal B(H)=S_\infty(\ell_2^n\otimes_2\h)$, $T$ can be naturally identified with a contraction $\hat{T}=\sum_{x,a}e_{a,x}\otimes T(e_{a,x})$ so that $$\hat{T}^\dag\hat{T}=\sum_{x,x'}e_{x,x'}\otimes \Big(\sum_aT(e_{a,x})^\dag T(e_{a,x'})\Big)\leq \uno_{S_\infty^n\otimes_{min}\mathcal B(\h)}.$$ In particular, it must hold that $\sum_aE_x^a\leq \uno_{\mathcal B(H)}$ for every $x$. Moreover, since the elements $E_x^a$ are semidefinite positive, the previous condition implies that, for every $x$, the linear map $T_x:\ell_\infty^n\rightarrow \mathcal B(H)$, defined as $T_x(e_a)=E_x^a$ for every $a$, is a completely positive and (completely) contractive map. Thus, $\sum_a \alpha_aE_x^a=T_x(\sum_a\alpha_a e_a)$ is a contraction in $\mathcal B(H)$.
  \end{proof}
  
 The proof of the following lemma is inspired by the analysis of the entangled value of the CHSH$_q$ games conducted in \cite{Ba15}. We define $e_{xyab}$ as $e_{xyab}=e_x\otimes e_y\otimes e_a\otimes e_b$.
  \begin{lemma}\label{lemma_upper_bound_eta_n}
  Let $n$ be a prime number and consider 
  \begin{align}\label{def_eta}
\eta_n=\sum_{\substack{x,y,a,b\in \mathbb Z_n
    \\ \hspace{-0.2 cm}a+b=xy}}e_{xyab,1}\otimes e_{a,x}\otimes e_{b,y}\in C_{n^4}\otimes S_1^n\otimes S_1^n.
    \end{align} Then,
    $$\|\eta_n\|_{C_{n^4}\otimes_{min} (S_1^n\otimes_{max} S_1^n)}\leq \sqrt{2}n^{\frac{3}{4}}.$$
  \end{lemma}
  \begin{proof}
 According to Equation (\ref{Def_max_norm}), we must show that for every pair of complete contractions $T:S_1^n\rightarrow \mathcal B(\h)$ and $S:S_1^n\rightarrow \mathcal B(\h)$ satisfying  $[T(x), S(y)]=[T(x), S(y)^\dag]=0$ for every $x,y\in S_1^n$, we have 
 \begin{align*}
 \|(\text{Id}\otimes T\odot S)(\eta_n)\|^2_{C_{n^4}\otimes_{min} (S_1^n\otimes_{max} S_1^n)}&\leq 2n\sqrt{n}. 
    \end{align*}
  
Then, for any such pair of complete contractions, it holds that 
 \begin{align*}
 \|(\text{Id}\otimes T\odot S)(\eta_n)\|^2_{C_{n^4}\otimes_{min} (S_1^n\otimes_{max} S_1^n)}&=\Big\|\sum_{\substack{x,y,a,b\in \mathbb Z_n
    \\ \hspace{-0.2 cm}a+b=xy}}e_{xyab,1}\otimes T(e_{a,x}) S(e_{b,y})\Big\|^2_{C_{n^4}\otimes_{min} \mathcal B(\h)} \\&
    =\Big\|\sum_{\substack{x,y,a,b\in \mathbb Z_n
    \\ \hspace{-0.2 cm}a+b=xy}}T(e_{a,x})^\dag T(e_{a,x}) S(e_{b,y})^\dag S(e_{b,y})\Big\|_{ \mathcal B(\h)},
\end{align*}where we have used the commutativity relations between $T$ and $S$.

If we denote $E_x^a=T(e_{a,x})^\dag T(e_{a,x})$, $F_y^b=S(e_{b,y})^\dag S(e_{b,y})$ and $\omega=e^{\frac{2\pi i}{n}}$, we can write the previous expression as 
 \begin{align*}
\Big\|\frac{1}{n}\sum_{k=0}^{n-1}\sum_{x,y,a,b\in \mathbb Z_n}\omega^{k (a+b-xy)}E_x^a F_y^b\Big\|_{ \mathcal B(\h)} =
    \frac{1}{n}\Big\|\sum_{k=0}^{n-1}\sum_{x,y\in \mathbb Z_n}\omega^{-kxy}\Big(\sum_a \omega^{ka}E_x^a\Big) \Big(\sum_b \omega^{kb}F_y^b\Big)\Big\|_{\mathcal B(\h)}. 
\end{align*}

According to Lemma \ref{lemma_basic_1}, the complete contractivity of $T$ and $S$ guarantees that $T_x^k:=\sum_a \omega^{ka}E_x^a$ and $S_y^k:=\sum_b \omega^{kb}F_y^b$ are contractions in $ \mathcal B(\h)$ for every $x$, $y$, $k$. 
On the other hand, for every norm-one vectors $\xi,\, \eta \in H$, we have 
\begin{align*}
\frac{1}{n}\sum_{k=0}^{n-1}\sum_{x,y\in \mathbb Z_n}\omega^{-kxy}\langle \eta, (T_x^k S_y^k)(\xi)\rangle =\frac{1}{n}\sum_{k=0}^{n-1}\langle v_k,(\text{Id}\otimes \Phi_k)(u_k)\rangle,  
\end{align*}where $u_k=\sum_y S_y^k(\xi)\otimes e_y\in \h \otimes \ell_2^n$, $v_k=\sum_x(T_x^k)^\dag(\eta)\otimes e_x\in \h \otimes \ell_2^n$ and $\Phi_k=(\omega^{-kxy})_{x,y}\in M_n$.

Now, it is clear that $\|u_k\|\leq \sqrt{n}$ and $\|v_k\|\leq \sqrt{n}$ for every $k$. Also, $\|\Phi_k\|=\sqrt{n}$ for every $k\neq 0$ (note that $(1/\sqrt{n})\Phi_k$ is a unitary matrix) and $\|\Phi_0\|=n$. Hence, we obtain 
\begin{align*}
\frac{1}{n}\sum_{k=0}^{n-1}\langle v_k,(\text{Id}\otimes \Phi_k)(u_k)\rangle \leq n+ (n-1)\sqrt{n}\leq 2n\sqrt{n}.
\end{align*}
  \end{proof}
  
The following lemma will be very useful in  order to prove the lower bound in Theorem \ref{Thm_Main_I}.
\begin{lemma}\label{lemma_basic_2}
For any natural number $n\geq 2$, let's consider a function $f:\mathbb Z_n^4\rightarrow \{0,1\}$ such that for every $x$, $y$, there exists a bijection $\pi_{x,y}:\mathbb Z_n\rightarrow \mathbb Z_n$ such that $f(x,y,a,b)=1$ if and only if $\pi_{x,y}(a)=b$. Then, 
$$\xi_f=\sum_{x,y,a,b\in \mathbb Z_n}f(x,y,a,b)\,  e_{1,xyab}\otimes e_{a,x}\otimes e_{b,y}$$satisfies 
$$\|\xi_f\|_{R_{n^4}\otimes_h(S_\infty^n\otimes_h S_\infty^n)}= {n^2}.$$ 

The same estimate applies to the element $$\xi^T_f=\sum_{x,y,a,b\in \mathbb Z_n}f(x,y,a,b)\,  e_{xyab}\otimes e_{b,y}\otimes e_{a,x}.$$
 \end{lemma} 
 \begin{proof}
 The homogeneity of the row operator spaces structure, together with the associativity of the Haagerup tensor norm and the completely isometric identification $S_\infty^n=C_n\otimes_hR_n $, allow us to rearrange the tensors $e_{xyab}=e_x\otimes e_y\otimes e_a\otimes e_b$ to write 
 $$\|\xi_f\|_{R_{n^4}\otimes_h(S_\infty^n\otimes_h S_\infty^n)}=\Big\|\sum_{y}e_y\otimes \xi_y\otimes e_{y}\Big\|_{R_{n}\otimes_h Y\otimes_hR_n},$$where $Y=R_{n^3}\otimes_hS_\infty^n\otimes_hC_n$ and $$\xi_y=\sum_{x,a,b}f(x,y,a,b)e_{xab}\otimes e_{a,x}\otimes e_b\in Y.$$
 
According to Lemma \ref{lemma_haag_P_1}, we have 
  $$\|\xi_f\|_{R_{n^4}\otimes_h(S_\infty^n\otimes_h S_\infty^n)}=\Big(\sum_y \|\xi_y\|_Y^2\Big)^{\frac{1}{2}}.$$
  
  In addition, for any fixed $y$ we can again use the properties of the Haagerup tensor norm to write 
  $$\|\xi_y\|_Y=\Big\|\sum_{b}e_b\otimes \xi_{y,b}\otimes e_{b}\Big\|_{R_{n}\otimes_h Z\otimes_hC_n},$$where $Z=R_{{n^2}}\otimes_hS_\infty^n$ and 
  $$\xi_{y,b}=\sum_{x,a}f(x,y,a,b)e_{xa}\otimes e_{a,x}\in Z.$$
  
 Now, according to Lemma \ref{lemma_haag_P_2}, we have 
   $$\|\xi_y\|_Y=\sum_b \Big\|\sum_{x,a}f(x,y,a,b)e_{xa}\otimes e_{a,x}\Big\|_Z.$$
   
 Finally, by once again invoking Lemma \ref{lemma_haag_P_1}, we find that for ever $y$ and $b$, we have 
\begin{align*}
\Big\|\sum_{x,a}f(x,y,a,b)e_{xa}\otimes e_{a,x}\Big\|_Z&= \Big\|\sum_{x}e_x\otimes \Big(\sum_{a}f(x,y,a,b)e_a\otimes e_a\Big)\otimes e_x \Big\|_{R_n\otimes_h(S_1^n)\otimes_h R_n}\\&=\Big(\sum_x\Big\|\sum_af(x,y,a,b)e_a\otimes e_a\Big\|_{S_1^n}^2\Big)^{\frac{1}{2}}=\sqrt{n},
\end{align*}where the last equality follows from the properties of the function $f$.

It then follows that $\|\xi_f\|_{R_{n^4}\otimes_h(S_\infty^n\otimes_h S_\infty^n)}=n^2$ as we wanted.

The estimate for $\xi^T_f$ can be proved completely analogously by just interchanging $x$,$y$ and $a$,$b$.
 \end{proof}

We are now ready to prove Theorem \ref{Thm_Main_I}. To do so, we will use notation $\h_A=\h_B=\ell_2^n$ to indicate the order considered when taking the Haagerup tensor norm.
\begin{proof}[Proof of Theorem \ref{Thm_Main_I}]
 According to Lemma \ref{lemma_upper_bound_eta_n}, for any prime number $n$ the element 
  \begin{align*}
\eta_n=\sum_{\substack{x,y,a,b\in \mathbb Z_n
    \\ \hspace{-0.2 cm}a+b=xy}}e_{xyab,1}\otimes e_{a,x}\otimes e_{b,y}
    \end{align*} satisfies
    $$\|\eta_n\|_{C_{n^4}\otimes_{min} (S_1(H_A)\otimes_{max} S_1(H_B))}\leq \sqrt{2}n^{\frac{3}{4}}.$$
    
    On the other hand, according to Lemma \ref{lemma_basic_2} applied to the function $f(x,y,a,b)=\delta_{b,xy-a}$,  the element $$P_n=\frac{1}{{n^2}}\sum_{\substack{x,y,a,b\in \mathbb Z_n
    \\ \hspace{-0.2 cm}a+b=xy}}e_{1,xyab}\otimes e_{a,x}\otimes e_{b,y}$$satisfies 
\begin{align}\label{Norm_P}
\max\{\|P_n\|_{R_{n^4}\widehat{\otimes} (S_\infty(H_A)\otimes_h S_\infty(H_B))}, \|P_n^T\|_{R_{n^4}\widehat{\otimes} (S_\infty(H_B)\otimes_h S_\infty(H_A))} \}= 1.
\end{align}  

Hence, for any decomposition $\eta_n=\eta_{n,1}+\eta_{n,2}$ in $C_{n^4}\otimes S_1(H_A)\otimes S_1(H_B)$ we have that
\begin{align*}
n&=\langle \eta_n, P_n\rangle=\langle \eta_{n,1}, P_n\rangle+\langle \eta_{n,2}, P_n\rangle=\langle \eta_{n,1}, P_n\rangle+\langle \eta_{n,2}^T, P_n^T\rangle\\& \leq \|\eta_{n,1}\|_{C_{n^4}\otimes_{min} (S_1(H_A)\otimes_{h} S_1(H_B))}+\|\eta_{n,2}^T\|_{C_{n^4}\otimes_{min} (S_1(H_B)\otimes_{h} S_1(H_A))}.
\end{align*}
According to the definition of the $\mu$ norm (\ref{Def_mu_norm}), this concludes the proof. 
\end{proof}

\section{Proof of Theorem \ref{thm2}}

In this section we prove the following theorem, from which Theorem \ref{thm2} follows immediately.
\begin{theorem}\label{Thm_Main_II}
For every prime number $n$ there exists an element $\beta_n\in M_{n^6}\otimes \ell_1^{n^2}\otimes \ell_1^{n^2}$ such that $\|\beta_n\|_{S_\infty^{n^6}\otimes_{min} (\ell_1^{n^2}\otimes_{max} \ell_1^{n^2})}\leq \sqrt{2}n^{\frac{3}{4}}$ and $\|\beta_n\|_{S_\infty^{n^6}\otimes_{min} (\ell_1^{n^2}\otimes_{\mu} \ell_1^{n^2})}\geq  n$.
\end{theorem}

The strategy to prove Theorem \ref{thm2} consists of transforming the element  $\eta_n\in C\otimes S_1\otimes S_1$ defined in the proof of Theorem \ref{thm1}, into another element $\beta_n\in S_\infty\otimes \ell_1\otimes \ell_1$  such that the norms $max$ and $\mu$ are preserved. To do this, given $n$, let us consider the unitary matrices $X,\, Z\in M_n$, defined by $X(e_j)=e^{\frac{2\pi i j}{n}}e_j$ and $Z(e_j)=e_{j+1}$ respectively on the canonical vectors $(e_j)_{j=0}^{n-1}$ and denote $T_{k,l}=X^kZ^l$ for every $k,l=0,1, \ldots, n-1$. One can easily check that
\begin{align}\label{flipT}
\frac{1}{n}\sum_{k,l=1}^{n-1}T_{k,l}\otimes T_{k,l}^\dag=\text{flip}: x\otimes y\mapsto y\otimes x.
\end{align}

The following lemma, whose proof is included for completeness, was proved in \cite{Ju16} and is inspired by the quantum teleportation protocol in quantum information theory. Since we have already used the notation $e_{i,j}=e_i\otimes e_j$ for the canonical basis of $M_n$, we will use the notation $f_{i,j}$ to denote the canonical basis of $\ell_\infty^{n^2}$ (or $\ell_1^{n^2}$).
For an operator algebraic interpretation of the quantum teleportation protocol we refer the interested reader to \cite{HaJuGaoTeleportation}. 
\begin{lemma}\label{lem_telep}
Let $J:S_1^n\rightarrow S_\infty^n(\ell_1^{n^2})$ and $W:S_\infty^n\rightarrow  S_\infty^n(\ell_\infty^{n^2})$ be the linear maps defined, respectively, as $$J(\rho)= \frac{1}{n}\sum_{k,l=0}^{n-1}T_{k,l}^\dag\rho T_{k,l} \otimes f_{k,l}\hspace{0.3 cm} \text{and} \hspace{0.3 cm} W(\rho)=\sum_{k,l=0}^{n-1}T_{k,l}^T\rho \overline{T}_{k,l} \otimes f_{k,l}.$$

Then,  $J$ and $W$ are complete contractions.
\end{lemma}
\begin{proof}
Since $W$ is a completely positive and unital map between $C^*$-algebras, we have $\|W\|_{cb}=1$.

In order to study the map $J$, let us define  $\eta_{k,l}=(T_{k,l}\otimes \text{Id})(\phi)\in \ell_2^{{n^2}}$ for every $k,l=0,\ldots , n-1$, where $\phi=(1/\sqrt{n})\sum_{j=0}^{n-1}e_j\otimes e_j$. It is very easy to check that these vectors form an orthonormal basis of $\ell_2^{{n^2}}$. Moreover, Equation (\ref{flipT}) implies that for every $h\in \ell_2^n$, the following identity holds $$h\otimes \sum_{i=0}^{n-1}e_i\otimes e_i=\frac{1}{\sqrt{n}}\sum_{k,l=0}^{n-1}\eta_{k,l}\otimes T_{k,l}^\dag(h).$$

This allows us to write, for every $\rho\in S_1^n$, the identity $$\rho\otimes \phi=\frac{1}{n}\sum_{k,l,k',l'=0}^{n-1}H_{k,l}^{k',l'}\otimes T_{k,l}^\dag\rho T_{k',l'},$$where $\phi=\sum_{i,j=0}^{n-1}e_{i,j}\otimes e_{i,j}$ and $H_{k,l}^{k',l'}:M_{{n^2}}\rightarrow M_{{n^2}}$ is the rank one operator defined as $H_{k,l}^{k',l'}(v)=\langle \eta_{k',l'}, v\rangle \eta_{k,l}$.

On the other hand, the fact that the vectors $\eta_{k,l}$, $k,l=0,1, \ldots, n-1$ form an orthonormal basis of $\ell_2^{{n^2}}$ guarantees that the projection $P:S_1^{{n^2}}\rightarrow \ell_1^{n^2}$, defined by $$P(A)=\sum_{k,l=0}^{n-1} \langle \eta_{k,l}, A(\eta_{k,l})\rangle {f_{k,l}},$$ is a complete contraction. According to Lemma \ref{Lemma_Aux_I}, the identity map $\text{Id}: S_1^n\hat{\otimes}(S_1^n\otimes_{min} S_\infty^n)\rightarrow  S_1^{{n^2}}\otimes_{min} S_\infty^n$ is completely contractive. It then follows that the map $i:S_1^n\rightarrow S_1^{{n^2}}\otimes_{min} S_\infty^n$, defined by $i(\rho)=\rho\otimes \phi$, is automatically completely contractive and the complete contractivity of $J$ follows from the identity $J=(P\otimes \text{Id})\circ i$.
\end{proof}

Lemma \ref{lem_telep} is the key point to transfer the statement of Theorem  \ref{Thm_Main_I} on non-commutative $L_1$-spaces into a statement involving commutative ones. Indeed, if $\eta_n$ is the element defined in (\ref{def_eta}), then we define $\beta_n=(\text{Id}\otimes J\otimes J)(\eta_n)$, so that
\begin{align}\label{Def beta}
\beta_n=\frac{1}{n^2}\sum_{\substack{x,y,a,b\in \mathbb Z_n
    \\ \hspace{-0.2 cm}a+b=xy\\ k,l,k',l'\in \mathbb Z_n}}e_{xyab,1}\otimes R_{a,x}^{k,l}\otimes R_{b,y}^{k',l'}\otimes {f_{k,l}}\otimes {f_{k',l'}} \in C_{n^4}\otimes S_\infty^n\otimes S_\infty^n\otimes (\ell_1^{n^2}\otimes \ell_1^{n^2}),
    \end{align}where we have denoted $R_{a,x}^{k,l}=T_{k,l}^\dag e_{a,x} T_{k,l}$ for every $a,x,k,l$.
    
    \begin{remark}\label{Remark_rearrangement}
  { The element $\beta_n=(\text{Id}\otimes J\otimes J)(\eta_n)$ naturally belongs to the space $C_{n^4}\otimes S_\infty^n(\ell_1^{n^2})\otimes S_\infty^n(\ell_1^{n^2})$. However, from an algebraic perspective, one can rearrange the terms in the tensor to represent this element in  $C_{n^4}\otimes S_\infty^n\otimes S_\infty^n\otimes (\ell_1^{n^2}\otimes \ell_1^{n^2})$.  As we will see, this rearrangement works correctly with respect to the estimates of the tensor norms that we need to study}. 
    \end{remark}
    
 Note that the complete contractivity of $J$ along with the commutative and associative properties of the min norm\footnote{{These properties allow us to rearrange the tensor in the way mentioned in Remark \ref{Remark_rearrangement}.}} and the metric mapping property immediately imply that 
 $$\|\beta_n\|_{C_{n^4}\otimes_{min} S_\infty^{n^2}\otimes_{min} (\ell_1^{n^2}\otimes_{min} \ell_1^{n^2})}\leq \|\eta_n\|_{C_{n^4}\otimes_{min} (S_1^n\otimes_{min} S_1^n)}.$$
 
 The next lemma shows that we can actually upper bound the max norm of $\beta_n$. 
 \begin{lemma}\label{lemma: lower bound theorem 2}
Given any natural number $n$, we have
 $$\|\beta_n\|_{C_{n^4}\otimes_{min} S_\infty^{n^2}\otimes_{min} (\ell_1^{n^2}\otimes_{max} \ell_1^{n^2})}\leq \|\eta_n\|_{C_{n^4}\otimes_{min} (S_1^n\otimes_{max} S_1^n)}.$$
 \end{lemma}
 \begin{proof}
 Given any pair of (complete) contractions $T,S: \ell_1^{n^2}\rightarrow  \mathcal B(\h)$ such that $[T({f_{k,l}}), S({f_{k',l'}})]=[T({f_{k,l}}), S({f_{k',l'}})^\dag]=0$ for every $k,l,k',l'$, we define the linear maps $$\tilde{T}, \tilde{S}:M_n(\ell_1^{n^2})\rightarrow M_n\otimes_{min}M_n\otimes_{min} \mathcal B(\h)=\mathcal B(\ell_2^n\otimes_2 \ell_2^n\otimes_2\h),$$as $\tilde{T}\Big(\sum_{k,l}A_{k,l}\otimes {f_{k,l}}\Big)=\sum_{k,l}A_{k,l}\otimes \uno\otimes T({f_{k,l}})$ and $\tilde{S}\Big(\sum_{k,l}B_{k,l}\otimes {f_{k,l}}\Big)=\sum_{k,l}\uno\otimes B_{k,l}\otimes S({f_{k,l}})$. 
 
 Since $\tilde{T}$ and $\tilde{S}$ are clearly complete contractions, we deduce from Lemma \ref{lem_telep} that the maps $$\hat{T},\hat{S}:S_1^n\rightarrow \mathcal B(\ell_2^n\otimes_2 \ell_2^n\otimes_2\h),$$ defined as $\hat{T}=\tilde{T}\otimes J$ and $\hat{S}=\tilde{S}\otimes J$ respectively, are also completely contractive.  Moreover, it is clear that $[\hat{T}(\rho), \hat{S}(\gamma)]=[\hat{T}(\rho), \hat{S}(\gamma)^{\dag}]=0$ for every $\rho,\, \gamma\in S_1^n$. Then, we conclude our proof by noticing that
 \begin{align*}
& \|(\text{Id}_{C_{n^4}\otimes S_\infty^{n^2}}\otimes T\odot S)(\beta_n)\|_{C_{n^4}\otimes_{min} S_\infty^{n}\otimes_{min} S_\infty^{n}\otimes_{min} \mathcal B(\h)}\\&= \|(\text{Id}_{C_{n^4}}\otimes \hat{T}\odot \hat{S})(\eta_n)\|_{C_{n^4}\otimes_{min} \mathcal B(\ell_2^n\otimes_2 \ell_2^n\otimes_2\h)}\\&\leq \|\eta_n\|_{C_{n^4}\otimes_{min} (S_1^n\otimes_{max} S_1^n)}.
 \end{align*}
\end{proof}

We are now ready to prove the main result of the paper. It's very useful to denote  $\h_A=\h_B=\ell_2^n$ as well as $\ell_\infty^{{n_A}^2}=\ell_\infty^{{n_B}^2}=\ell_\infty^{{n}^2}$ to precisely indicate the rearrangement in the various steps of the proof.
\begin{proof}[Proof of Theorem \ref{thm2}]

On the one hand, according to Lemma \ref{lemma: lower bound theorem 2}, we have that 
 $$\|\beta_n\|_{C_{n^4}\otimes_{min} S_\infty(H_A)\otimes_{min}S_\infty(H_B)\otimes_{min} (\ell_1^{n_A^2}\otimes_{max} \ell_1^{n_B^2})}\leq \|\eta_n\|_{C_{n^4}\otimes_{min} (S_1(H_A)\otimes_{max} S_1(H_B))}\leq \sqrt{2}n^{\frac{3}{4}}.$$

On the other hand, using the notation of Lemma \ref{lem_telep}, let us consider the element 
\begin{align}\label{Def Q}
Q_n=\frac{1}{n^2}(\text{Id}_{R_{n^4}}\otimes W\otimes W)(P_n)=\frac{1}{n^4}\sum_{\substack{x,y,a,b\in \mathbb Z_n
    \\ \hspace{-0.2 cm}a+b=xy\\ k,l,k',l'\in \mathbb Z_n}}e_{1,xyab}\otimes S_{a,x}^{k,l}\otimes S_{b,y}^{k',l'}\otimes {f_{k,l}}\otimes {f_{k',l'}},
    \end{align}in $R_{n^4}\otimes S_\infty(H_A)\otimes S_\infty(H_B) \otimes (\ell_\infty^{n_A^2}\otimes \ell_\infty^{n_B^2})$, where we have denoted $S_{a,x}^{k,l}=T_{k,l}^T e_{a,x} \overline{T}_{k,l}$ for every $a,x,k,l$. {Note that once again, we are implicitly considering a rearrangement of the terms in the tensor $Q_n$ as explained in Remark \ref{Remark_rearrangement}.}
    
 Since $\langle \beta_n , Q_n\rangle=\langle \eta_n, P_n\rangle=n$, we can conclude the proof of the theorem in the same way as the proof of Theorem \ref{thm1}, by simply demonstrating that
\begin{align*}
\max\left\{\|Q_n\|_{(R_{n^4}\widehat{\otimes} S_1(H_A)\widehat{\otimes} S_1(H_B))\widehat{\otimes}(\ell_\infty^{{n_A}^2}\otimes_h\ell_\infty^{{n_B}^2})}, \|Q_n^T\|_{(R_{n^4}\widehat{\otimes} S_1(H_A)\widehat{\otimes} S_1(H_B))\widehat{\otimes}(\ell_\infty^{{n_B}^2}\otimes_h\ell_\infty^{{n_A}^2})} \right\}\leq 1.
\end{align*} 

In order to prove the inequality $$\|Q_n\|_{(R_{n^4}\widehat{\otimes} S_1(H_A)\widehat{\otimes} S_1(H_B))\widehat{\otimes}(\ell_\infty^{{n_A}^2}\otimes_h\ell_\infty^{{n_B}^2})}\leq 1,$$ we used that $\|P_n\|_{R_{n^4}\widehat{\otimes} (S_\infty(H_A)\otimes_h S_\infty(H_B))}\leq 1$ and the fact that $W$ is a complete contraction to guarantee that $$\|(\text{Id}_{R_{n^4}}\otimes W\otimes W)(P_n)\|_{R_{n^4}\widehat{\otimes} \big[(S_\infty(H_A)\otimes_{min}\ell_\infty^{{n_A}^2})\otimes_h (S_\infty(H_B)\otimes_{min}\ell_\infty^{{n_B}^2})\big]}\leq 1.$$

Now, according to the first property proved in Lemma \ref{Lemma_Aux_II} applied to $X=S_\infty(H_A)\otimes_{min}\ell_\infty^{{n_A}^2}$ and $Y=\ell_\infty^{{n_B}^2}$, the previous estimate implies that
$$\|(\text{Id}_{R_{n^4}}\otimes W\otimes W)(P_n)\|_{R_{n^4}\widehat{\otimes} \Big[S_\infty(H_B)\otimes_{min}\big(\big(S_\infty(H_A)\otimes_{min}\ell_\infty^{{n_A}^2}\big)\otimes_h \ell_\infty^{{n_B}^2}\big)\Big]}\leq 1,$$where the terms must be rearranged according to the map $T$ in Lemma \ref{Lemma_Aux_II}.

We can now apply the second property proven in Lemma \ref{Lemma_Aux_II} to $X=\ell_\infty^{{n_A}^2}$ and $Y=\ell_\infty^{{n_B}^2}$ to conclude that the previous estimate implies that
$$\|(\text{Id}_{R_{n^4}}\otimes W\otimes W)(P_n)\|_{R_{n^4}\widehat{\otimes} \Big[S_\infty(H_B)\otimes_{min}S_\infty(H_A)\otimes_{min}\big(\ell_\infty^{{n_A}^2}\otimes_h \ell_\infty^{{n_B}^2}\big)\Big]}\leq 1.$$

Finally, by using that $S_\infty(H_A)=S_\infty(H_B)=S_\infty^n$ and Equation (\ref{Id_com_int}), we can conclude that 
$$\|Q_n\|_{R_{n^4}\widehat{\otimes} S_1(H_B)\widehat{\otimes}S_1(H_A)\widehat{\otimes}\Big(\ell_\infty^{{n_A}^2}\otimes_h \ell_\infty^{{n_B}^2}\Big)}\leq 1.$$Note that the commutativity of the projective tensor norms implies that 
$$\|Q_n\|_{R_{n^4}\widehat{\otimes} S_1(H_A)\widehat{\otimes}S_1(H_B)\widehat{\otimes}\Big(\ell_\infty^{{n_A}^2}\otimes_h \ell_\infty^{{n_B}^2}\Big)}\leq 1$$as we wanted.

Given that $\|P_n^T\|_{R_{n^4}\widehat{\otimes} (S_\infty(H_B)\otimes_h S_\infty(H_A))} \leq 1$, employing the same argument enables us to deduce that $$\|Q_n^T\|_{R_{n^4}\widehat{\otimes} S_1(H_A)\widehat{\otimes}S_1(H_B)\widehat{\otimes}\Big(\ell_\infty^{{n_B}^2}\otimes_h \ell_\infty^{{n_A}^2}\Big)}\leq 1.$$

This concludes the proof.
\end{proof}

{\begin{remark}
A natural question that arises from the results in this work is whether Theorem \ref{thm2} remains valid if the operator space $S_\infty(\ell_2)$ is replaced by the space $C$. It seems clear that an affirmative answer to this question would require a completely different approach from the one followed in this work.
\end{remark}

\begin{remark}\label{rem_commuting}
Let $(x_i)_i$ and $(y_j)_j$ be two families of contractions in $\mathcal B(H)$ such that $[x_i,y_j]=0$ for every $i,j$. If we assume, in addition, the elements $x_i$ to be normal, it follows from Fuglede's Theorem \cite{Fu50} that $[x_i,y_j^\dag]=0$ for every $i,j$.  {Hence}, according to Russo-Dye Theorem (or by just using twice the standard block encoding; that is, given any contraction $T\in \mathcal B(H)$, we consider the unitary 
 \[ U \lel\left( \begin{array}{cc} T & -\sqrt{1-TT^\dag}   \\
  \sqrt{1-T^\dag T} & T^\dag \end{array} \right)\in M_2 (\mathcal B(H))), \] we obtain {that for every families of contractions $(x_i)_i$ and $(y_j)_j$ such that the elements $x_i$ are normal} and for every family of matrices $(a_{ij})_{i,j}\in M_n$, it holds that
 \[ \Big\|\sum_{i,j} a_{ij}\otimes x_i y_j \Big\|_{M_n(\mathcal B(H))}
 \kl \sup \Big\|\sum_{i,j} a_{ij} \otimes u_iv_j \Big\|_{M_n(\mathcal B(H))}
\leq  \Big \|\sum_{ij} a_{ij} \otimes {f_i}\otimes {f_j} \Big\|_{M_n(\ell_1\otimes_{\max}\ell_1)},  \pl \]where the supremum runs over all families of commuting unitaries $(u_i)_i$, $(v_j)_j$ in $\mathcal B(H)$. {Ricard demonstrated that the normality assumption of the elements $x_i$ is key in the previous inequality.} More precisely, he proved \cite[Proposition 13.10]{Pi20} the existence of families of {(non-normal)} commuting contractions $(x_i)_i$, $(y_j)_j$ and (even scalar) coefficients $(a_{ij})_{i,j}\in \mathbb C$ satisfying 
 \[ \Big\|\sum_{i,j} a_{ij}\otimes x_i y_j \Big\|_{M_n(\mathcal B(H))}
 >   \Big \|\sum_{ij} a_{ij} \otimes {f_i}\otimes {f_j} \Big\|_{M_n(\ell_1\otimes_{\max}\ell_1)}.  \pl \] {In this framework}, Theorem \ref{thm2}  provides families of commuting contractions that dramatically violate this inequality - such that even an equivalence between the norms is not possible. Note that considering matrix coefficients $(a_{ij})_{i,j}$ is crucial here, since Grothendieck's inequality (\ref{GRO}) implies that 
 \[ \Big\|\sum_{i,j} a_{ij} x_i y_j \Big\|_{\mathcal B(H)}\leq K_G
  \Big \|\sum_{ij} a_{ij}  {f_i}\otimes {f_j} \Big\|_{\ell_1\otimes_{\max}\ell_1},  \pl \] for every families of (non-necessarily commuting) contractions $(x_i)_i$ and $(y_j)_j$. Hence, Theorem \ref{thm2} can be understood as a qualitative witness to the fact that the commuting contractions $(x_i)_i$ and $(y_j)_j$ cannot be normal. We are grateful to Mikael R$\o$rdam for reminding us of Fuglede's result.
 \end{remark}

\end{document}